\documentclass[11pt,a4paper]{article}
\usepackage[latin1]{inputenc}
\usepackage{amsthm}
\usepackage{amsmath}
\usepackage{amsfonts}
\usepackage{amssymb}
\usepackage{graphicx}
\usepackage{textcomp}
\numberwithin{equation}{section}

\newtheorem{theorem}{Theorem}

\author {Fügen TORUNBALCI AYDIN\thanks{Corresponding author. Yildiz Technical University, Faculty of Chemical and Metallurgical Engineering, Department of Mathematical Engineering, 34220, Istanbul, Turkey.  \,e-mail:\,faydin@yildiz.edu.tr, ftorunay@gmail.com } \;, {Kevser KÖKLÜ\thanks{Yildiz Technical University, Faculty of Chemical and Metallurgical Engineering, Department of Mathematical Engineering, 34220, Istanbul, Turkey. \,\,e-mail:  ozkoklu@yildiz.edu.tr }}} 

\title{On Generalizations of the Pell Sequence}	
\begin{document}
\maketitle	

\begin{abstract}
	In this paper, we investigate the generalized Pell sequence, the generalized complex Pell sequence and the generalized dual Pell sequence using the Pell numbers. We obtain special cases of these sequences. Furthermore, we give recurrence relations, vectors, the silver ratio and Binet's formula for the above-investigated sequences.
		
	\textbf{Keywords:}   Pell number,  Pell-Lucas number, generalized Pell sequence, complex numbers, dual numbers.
	
\end{abstract}

\section{Introduction}
The Fibonacci, Lucas, Pell and Pell-Lucas numbers have important parts in mathematics. They are of fundamental importance in the fields of combinatorics and number theory (see, for example, \cite{horadam1, horadam4,koshy}). The use of such special sequences has increased significantly in quantum mechanics, quantum physics, etc.

The  Pell sequence
$$\,\;1\,,\,\,2\,,\,\,5,\,\,12,\,\,29,\,\,70,\,\,169,\,\,408,\,\,985,\,\,2378, \ldots \,,{P_n}, \ldots $$
is defined by the recurrence relation
$${P_n} = 2\,{P_{n - 1}} + {P_{n - 2}}\;,\;\;(n \ge 3),$$
with ${P_{1}} =1 , {P_{2}} = 2\,$, which is well known as the n-th term of the Pell sequence $({P_n})$. The Pell numbers were named after the English mathematician John Pell \cite{bicknell,horadam4,melham}.

The Pell - Lucas sequence

$\,\;2\,,\,\,6\,,\,\,14,\,\,34,\,\,82,\,\,198,\,\,478,\,\,1154,\,\,2786,\,\,6726, \ldots \,,{Q_n}, \ldots $
is defined by the recurrence relation
${Q_{n}} = 2\,{Q_{n - 1}} + {Q_{n - 2}}\;,\;\;(n \ge 3),$
with ${Q_{1}} =2 \,, {Q_{2}} = 6\,$, which is well known as the n-th term of the Pell-Lucas sequence $({Q_n})$ (François Edouard Anatole Lucas, 1876).\\ \\
Furthermore, we can found the matrix representations of the Pell and 
Pell-Lucas numbers in \cite{kilic,tasci}. \\ \,
\\ \,\,
For the Pell and Pell-Lucas numbers, the following properties

\begin{equation}\label {1-1}
\begin{array}{rl}
{P}_{m}{P}_{n + 1} +\,{P}_{m - 1}{P}_{n}=&{P}_{m+n} \,, \,\\
\\
{P}_{m}{P}_{n + 1} - {P}_{m + 1}{P}_{n}=&(-1)^n \,{P}_{m - n} \,, \,\\
\\
{P}_{n - 1}{P}_{n + 1} - {{P}_{n}^2}=&(-1)^{n} \,,\,\\
\\
{{P}_{n}^2} +\,{{P}_{n + 1}^2}=&{{P}_{2n + 1}} \,,\,\\
\\
{{P}_{n + 1}^2} -\,{{P}_{n - 1}^2}=&2\,{P}_{2n} \,,\,\\
\\
2\,{{P}_{n + 1}}\,{P}_{n} - 2\,{{P}_{n}^2}=&{P_{2n}} \,, \,\\
\\
{{P}_{n}^2} +\,{{P}_{n + 3}^2}=&5\,{{P}_{2n + 3}} \,,\,\\
\\
{{P}_{2n + 1}} + {{P}_{2n}}=&2\,{{P}_{n + 1}^2}- 2\,{{P}_{n}^2}- (-1)^n \,,\,\\
\\
{{P}_{n}^2} +\,{P}_{n - 1}{P}_{n + 1}=&\frac{{Q_{n}}}{4}  \,,\,\\
\\
{{P}_{n + 1}} + {P}_{n -1}=&{Q_{n}} \,,\,\\
\\
{P}_{n}\,{Q}_{n}=&{{P}_{2n}} \,\\\ 
\end{array}.
\end{equation}

and the following summation formulas
 
\begin{equation}\label{1-2} 
\begin{gathered}
  P_{1}^2+P_{2}^2+P_{1}^2+\,\dots\,+P_{n}^2=\frac{{P_n}P_{n+1}}{2},
\end{gathered}  
\end{equation}
\begin{equation}\label{1-3} 
\begin{gathered}
  \sum\limits_{k = 0}^{n}\binom{n}{k}\,2^k\,P_{2k}  =P_{2n},
\end{gathered}  
\end{equation}

\begin{equation}\label{1-4} 
\begin{gathered}
  \sum\limits_{k = 0}^{n}\binom{n}{k}\,P_{k}\,P_{n-k}  =2^{n}P_{n}
\end{gathered}  
\end{equation} 

 were hold \cite{bicknell,cimen,horadam4, melham}.  
\section{The Generalized Pell Sequence} 

In this section, we define the generalized Pell sequence denoted by ${\mathbb{P}_{n}}$. \\
\\
The generalized Pell sequence defined by
\begin{equation}\label{2-1}
{\mathbb{P}_{n}} = 2\,{\mathbb{P}_{n-1}} + {\mathbb{P}_{n-2}} \;,\;\;(n \ge 3)
\end{equation}
with ${\mathbb{P}_{0}} = q\,,\;{\mathbb{P}_{1}} = p\,,\;{\mathbb{P}_{2}} = 2\,p + \,q$,\,
where $p,\,q$ are arbitrary integers. That is, the generalized Pell sequence is 
\begin{equation}\label{2-2}
q\,,\,\,p \,,\,\,2\,p + q\,,\,\,5\,p + 2\,q\,,\,\,12\,p + 5\,q\,,\,\,29\,p + 12\,q\,, \,\ldots \,,(p - 2\,q){P_n} + q{P_{n + 1}}, \ldots
\end{equation}
Using the equations (\ref{2-1}) and (\ref{2-2}) , we get 
  \begin{equation}\label{2-3}
 \begin{aligned}
\begin{array}{lll}
	{\mathbb{P}_{n}} = p\,{P_{n}} +\,q\,{P_{n - 1}}\,,\\
	
  {\mathbb{P}_{n+1}} = p\,{P_{n + 1}} + q\,{P_{n}}\,,\\
	
  {\mathbb{P}_{n+2}} = (2\,p + q\,)\,{P_{n + 1}} + p \,{P_{n}}\-.
  \end{array}
   \end{aligned}
   \end{equation}

\noindent Putting ${n}={r}$ in the Eq.(\ref{2-3}) and using the Eq. (\ref{2-1}), we find in turn
\\
 \begin{equation}\label{2-4}
 \begin{aligned}
\begin{array}{lll}
  {\mathbb{P}_{r+3}} = (5\,p + 2\,q\,){P_{r+1}} +\,(2\,p +\,q\,)\,{P_{r}}={\mathbb{P}_{3}}\,{P_{r+1}}+{\mathbb{P}_{2}}\,{P_{r-1}}\,,\\
	
  {\mathbb{P}_{r+4}} = (12\,p + 5\,q)\,{P_{r+1}} + (5\,p + 2\,q)\,{P_{r-1}}={\mathbb{P}_{4}}\,{P_{r+1}}+{\mathbb{P}_{3}}\,{P_{r-1}}.\-
  \end{array}
   \end{aligned}
   \end{equation}
\\
Consequently, we obtain relation between the generalized Pell sequence and the Pell sequence as follows:

\begin{equation}\label{2-5}
\begin{aligned}
 {\mathbb{P}_{n+r}} = {\mathbb{P}_{n}}\,P_{r + 1} + {\mathbb{P}_{n-1}}\,P_{r}.
\end{aligned}
\end{equation}
\\
Also, certain results follow almost immediately from the Eq.(\ref{2-1})
\begin{equation}\label{2-6}
\begin{aligned}
{\mathbb{P}_{n+1}} - 5\, {\mathbb{P}_{n-1}} - 2\,{\mathbb{P}_{n-2}} = 0 \,,
\end{aligned}
\end{equation}
\begin{equation}\label{2-7}
\begin{aligned}
2\,{\mathbb{P}_{n}} - 4\, {\mathbb{P}_{n-1}} -\,{\mathbb{P}_{n-2}} = 0 \,,
\end{aligned}
\end{equation}
\begin{equation}\label{2-8} 
\begin{gathered}
  2\,\sum\limits_{i = 0}^{n-1} {\mathbb{P}_{2i+1}}  = {{\mathbb{P}_{2n}}-q },
\end{gathered}  
\end{equation}
\begin{equation}\label{2-9} 
\begin{gathered}
  2\,\sum\limits_{i = 1}^{n} {{\mathbb{P}_{2i}}}  = {{\mathbb{P}_{2n+1}}-p }.
\end{gathered}  
\end{equation} \\	
For the generalized Pell sequence, we have the following properties:\\
\begin{equation}
({\mathbb{P}_{n}})^2 +\,({\mathbb{P}_{n+1}})^2 = (2p - 2q){\mathbb{P}_{2n+1}} - e_{p}\,{P_{2n + 1}}\,,\;\;\;\;
\end{equation}
\begin{equation}
({\mathbb{P}_{n+1}})^2 -\,({\mathbb{P}_{n-1}})^2 = 2\,[ (2p - 2q){\mathbb{P}_{2n}}  - e_{p}\,{P_{2n}} ]\,,\;\;\;\;
\end{equation}
\begin{equation}
{\mathbb{P}_{n-1}}\,{\mathbb{P}_{n+1}} - ({\mathbb{P}_{n}})^2 = {( - 1)^n}\,e_{p}\,\,,\;\;\;\;
\end{equation}
\begin{equation}
({\mathbb{P}_{n+1}})^2 +\,e_{p}\,({P_{n})^2} = p\,{\mathbb{P}_{2n+1}}\,,\;\;\;\;
\end{equation}
\begin{equation}
{\mathbb{P}_{m}}\,{\mathbb{P}_{n+1}} -\,{\mathbb{P}_{m+1}}\,{\mathbb{P}_{n}} = {( - 1)^n}\, e_{p}\,{P_{m - n}}\,,\;\;\;\;
\end{equation}
\begin{equation}
{\mathbb{P}_{m}}\,{\mathbb{P}_{n+1}} +\,{\mathbb{P}_{m-1}}\,{\mathbb{P}_{n}} = (2p - 2q)\,{\mathbb{P}_{m+n}}\,- \,e_{p}\,P_{m + n}\,,
\end{equation}
\begin{equation}
{\mathbb{P}_{n+1-r}}\,{\mathbb{P}_{n+1+r}} -\,({\mathbb{P}_{n+1}})^2 = {( - 1)^{n - r}\, e_{p}\,{P_{r}^2\,}}\,,\;\;\;\;
\end{equation}
\begin{equation}
{\mathbb{P}_{n}}\,{\mathbb{P}_{n+r+1}} -\,{\mathbb{P}_{n-s}}\,{\mathbb{P}_{n+r+s+1}} = {( - 1)^{n + s}\, e_{p}\,{P_{s}\,{P_{r + s + 1}}}}\,,\;\;\;\;
\end{equation}
\begin{equation}
\frac{{\mathbb{P}_{n+r}} + ( - 1)^r{\mathbb{P}_{n-r}}} {{\mathbb{P}_{n}}} = {Q_{r}}\,,
\end{equation}\,
\\
where ${{e}_{p}} = {p}^{2} - 2\, p\,q -\,{q^2}$ \, $and$ \,${Q}_{r}$\,\,\textit{is the Pell-Lucas number}.\,\\
\\
In addition, we obtain the following relations
\begin{equation}
{P}_{n + r} + (-1)^r\,{P}_{n - r}={Q}_{r}\,{P}_{n} \,, \,\\
\end{equation}
\begin{equation}
{P}_{n + r}\,{P}_{n - r} - {{P}_{n}^2}=(-1)^{n-r+1}\,{{P}_{r}^2} \,, \,\\ 
\end{equation}
\begin{equation}
{P}_{n}\,{P}_{n + r + 1} - {P}_{n - s}\,{P}_{n + r + s + 1}=(-1)^{n+s}\,{P}_{s}\,{P}_{r + s + 1} \,, \,\\
\end{equation}
\begin{equation} 
{P}_{n - 1}\,{P}_{n + r} - {P}_{n - s - 1}\,{P}_{n + r + s}=(-1)^{n+s+1}\,{P}_{s}\,{P}_{r + s + 1} \,, \,\\
\end{equation}
\begin{equation}
\begin{array}{rl}
{{P}_{n}\,{P}_{n + r} + {P}_{n - 1}\,{P}_{n + r + 1} - {P}_{n - s}\,{P}_{n + r + s} - {P}_{n - s - 1}\,{P}_{n + r + s + 1}} \\
 \quad \quad =2\,(-1)^{n+s+1}\,{P}_{s}\,{P}_{r + s + 1} \,. \,\\
\end{array}
\end{equation}
\begin{theorem}
	If \, ${\mathbb{P}_{n}}$ is the generalized Pell number, then
	
	$$\mathop {\lim }\limits_{n \to \infty } \frac{{\mathbb{P}_{n+1}}}{{\mathbb{P}_{n}}} = \frac{p\,\alpha+\,q} {q\alpha  + (p - 2\,q) },$$
\\	
where $\alpha  = ({1 + \sqrt 2 )}$.
\end{theorem}

\begin{proof}
	We know for the Pell number $P_{n}$
	$$\mathop {\lim }\limits_{n \to \infty } \frac{P_{n + 1}}{P_n} = \alpha, $$
\\ \, 
where $\alpha  = ({1 + \sqrt 2 )}$ \cite{horadam4}. \\ \\ \,
Then for the generalized Pell number ${\mathbb{P}_{n}}$ ,\, 
we obtain \\
  \begin{equation} \label{2-24}
	\begin{aligned} 
	\begin{array}{rl}
	$$\mathop {\lim }\limits_{n \to \infty } \frac{{\mathbb{P}_{n+1}}}{{\mathbb{P}_{n}}} =& \mathop {\lim }\limits_{n \to \infty } \frac{{{p\,{P_{n + 1}}\, +\,q\,{P_{n}}}}} {{\;  q \,{P_{n + 1}} + (p - 2\,q)\,{P_{n}}}} \\ 
	\\
	     =& \frac {p\,\alpha  +\,q} {q\,\alpha +(p -2\,q)} \, . $$
\end{array}
\end{aligned}
\end{equation}
\end{proof}
\begin{theorem}
	The Binet's formula \footnote[2]{Binet's formula is the explicit formula to obtain the n-th Pell and Pell-Lucas numbers. It is well known that for the Pell and Pell-Lucas numbers, Binet's formulas are
		$$
		{P_n} = \frac{{\alpha ^n} - {\beta ^n}} {\alpha  - \beta }
		$$
		and  
		$$
		{Q_n} = {\alpha ^n} + {\beta ^n}
		$$
		respectively, where $\alpha  + \beta  = 2\,,\;\alpha  - \beta  = {2 \sqrt 2 } \,,\;\alpha \beta  =  - 1$    and $\alpha  = {1 + \sqrt 2}$ ,\, $\beta  = { 1 - \sqrt 2 }$,\,\,\cite{horadam4,walton}.} 
	for  the generalized Pell sequence is as follows;	
\begin{equation}
	{\mathbb{P}_{n}} = \frac{(\,\bar \alpha \,\,{\alpha ^n} - \bar \beta \,\,{\beta ^n})} {\alpha  - \beta }.
\end{equation}
\end{theorem}

\begin{proof}
Characteristic equation of the recurrence relation ${\mathbb{P}_{n+2}}={\mathbb{P}_{n+1}}+ 2\,{\mathbb{P}_{n}}$ \ \ is

\begin{equation}\label{2-26}
{{t}^{2}}- 2\,t-1=0.
\end{equation}
The roots of this equation are 

\begin{equation}\label{2-27}
\alpha  = {{1 + \sqrt 2 }} \, \, \, \, \text{and} \, \, \, \, \beta  = {{1 - \sqrt 2 }},
\end{equation} 
\\
where  $\alpha +\beta =2\ ,\ \ \alpha -\beta =2 \sqrt 2\,,\ \ \alpha \beta =-1$.\\
\\
Using the recurrence relation and ${\mathbb{P}_{0}}=q$,  ${\mathbb{P}_{1}}=p$  initial values, we obtain the Binet's formula for ${\mathbb{P}_{n}}$ as
	
\begin{equation*}\label{2-28}
	\begin{aligned}
	\begin{array}{lll}
	{\mathbb{P}_{n}} &= (\,p -\,2\,q\,)\,{P_n} +\,q\,{P_{n + 1}}\\
	& = (\,p  - \,2\,q\,)\,(\frac{{\alpha ^n} - \,{\beta ^n}} {\alpha  - \,\beta }) + \,q\,(\frac{{\alpha ^{n + 1}} - {\beta ^{n + 1}}} {\alpha  - \,\beta }) \\ 
	& = \frac{\overline \alpha  \;{\alpha ^n}\; - \,\overline {\beta \,} \,{\beta ^n}}  {\alpha  - \beta },\,
	
	\end{array}
	\end{aligned}
	\end{equation*}
\\
where $\bar {\alpha }=\,p + \, q\,({\alpha} - 2)\,, \, \, \, \,
\bar {\beta }=\,p + q\,({\beta}- 2)\,$.
\end{proof}

\subsection{The Generalized Pell Vectors}
 A generalized Pell vector is defined by
$$
\overrightarrow {{\mathbb{P}_{n}}} = ({\mathbb{P}_{n}}\,,\,{\mathbb{P}_{n+1}},\,{\mathbb{P}_{n+2}}).
$$
From the Eq. (\ref{2-2}) it can be expressed as

\begin{equation}\label{2-29}
\begin{aligned}
\begin{array}{lll}
\overrightarrow {{\mathbb{P}_{n}}}&=(p - 2\,q)\overrightarrow {{\rm{P}}}_{n}+\,q\overrightarrow {{\rm{P}}}_{\rm{n + 1}},
\end{array}
\end{aligned}
\end{equation}
where $\overrightarrow {{\mathbb{P}_{n}}} = ({\mathbb{P}_{n}}\,,\,{\mathbb{P}_{n+1}},\,{\mathbb{P}_{n+2}})$. \ $\overrightarrow {{\mathbb{P}_{n+1}}} = ({\mathbb{P}_{n+1}}\,,\,{\mathbb{P}_{n+2}},\,{\mathbb{P}_{n+3}})$ is also $(n+1)$-th Pell vector.\\
\\
The product of \,$\overrightarrow {{\mathbb{P}_{n}}}$ and $ \lambda \in \mathbb{R}$ is given by 
$$ 	\lambda \,\overrightarrow {{\mathbb{P}_{n}}} =(\lambda \overrightarrow {{\mathbb{P}_{n}}},\,\, \lambda\overrightarrow {{\mathbb{P}_{n+1}}},\,\, \lambda\overrightarrow {{\mathbb{P}_{n+2}}})  $$
and,  $\overrightarrow {{\mathbb{P}_{n}}}$ and $\overrightarrow {{\mathbb{P}_{m}}}$ are equal if and only if 
$$
\begin{aligned}
\begin{array}{lcl}
{\mathbb{P}_{n}}&=&{\mathbb{P}_{m}} \\
{\mathbb{P}_{n+1}} &= &{\mathbb{P}_{m+1}} \\
{\mathbb{P}_{n+2}}&=&{\mathbb{P}_{m+2}}.
\end{array}
\end{aligned} 
$$
\begin{theorem}
Let ${\overrightarrow {{\mathbb{P}_{n}}}}$ and ${\overrightarrow {{\mathbb{P}_{m}}}}$ be two generalized Pell vectors. The dot product of ${\overrightarrow {{\mathbb{P}_{n}}}}$ and            
${\overrightarrow {{\mathbb{P}_{m}}}}$ is given by

\begin{equation} \label{2-30}
\begin{array} {rl}
\left\langle \overrightarrow {{\mathbb{P}_{n}}},\overrightarrow {{\mathbb{P}_{m}}} \right\rangle  =& {p ^2}( {P_{n+m+3}}+{P_{n}}{P_{m}} ) \\ 
 &+\,p\,q\,( {P_{n+m+2}}+ 2\,{P_{n+m}}+ 2\,{P_{n+1}}\,{P_{m+1}}) \\ 
&+ {q^2}\,({P_{n+m+1}}+{P_{n-1}}\,{P_{m-1}}) \,.
\end{array}
\end{equation}
\end{theorem}
\begin{proof}
The dot product of $\;$
$\overrightarrow {{\mathbb{P}_{n}}}=({\mathbb{P}_{n}},\,{\mathbb{P}_{n+1}},\,{\mathbb{P}_{n+2}})$ \, \,and\,\\ 
$\overrightarrow {{\mathbb{P}_{m}}}=({\mathbb{P}_{m}}\,,\,{\mathbb{P}_{m+1}},\,{\mathbb{P}_{m+2}})$  \, defined by 
\begin{equation*}
\begin{aligned}
\begin{array} {lll}
\left\langle \overrightarrow {{\mathbb{P}_{n}}},\overrightarrow {{\mathbb{P}_{m}}} \right\rangle & = {\mathbb{P}_{n}} {\mathbb{P}_{m}}+{\mathbb{P}_{n+1}}{\mathbb{P}_{m+1}} + {\mathbb{P}_{n+2}}{\mathbb{P}_{m+2}}.\\
\end{array}
\end{aligned}
\end{equation*}
Using the equations (\ref{2-1}), (\ref{2-2}) and (\ref{2-3}), we obtain

\begin{equation} \label{2-31}
\begin{aligned}
\begin{array} {lll}
{\mathbb{P}_{n}}\,{\mathbb{P}_{m}}={p ^2}\,({P_{n}}\,{P_{m}}) + p\,q\, ({P_{n}}\,{P_{m - 1}} + {P_{n - 1}}\,{P_{m}})\\
\quad \quad \quad \quad \quad \quad + {q ^2}\,({P_{n - 1}}\,{P_{m - 1}}),
\end{array}
\end{aligned}
\end{equation}
\begin{equation} \label{2-32}
\begin{aligned}
\begin{array} {lll}
{\mathbb{P}_{n+1}}\,{\mathbb{P}_{m+1}}={p ^2}\,({P_{n+1}}\,{P_{m+1}})+p\,q\,({{P_{n+1}}{P_{m}}+{P_{n}}\,{P_{m+1}}})\\
\quad \quad \quad \quad \quad \quad \quad \quad +{q ^2}\,({P_{n}}\,{P_{m}}),
\end{array}
\end{aligned}
\end{equation}
\begin{equation} \label{2-33}
\begin{aligned}
\begin{array} {lll}
{\mathbb{P}_{n+2}}\,{\mathbb{P}_{m+2}}={p ^2}\,({P_{n+2}}\,{P_{m+2}})+p\,q\,({P_{n+2}}{P_{m+1}}+{P_{n+1}{P_{m+2}}})\\
\quad\quad \quad \quad \quad \quad \quad \quad +{q ^2}\,({P_{n+1}}{P_{m+1}}),
\end{array}
\end{aligned}
\end{equation}
\begin{equation*} \label{2-34}
\begin{aligned}
\begin{array} {rl}
\left\langle  \overrightarrow {{\mathbb{P}_{n}}},\overrightarrow {{\mathbb{P}_{m}}} \right\rangle =& {p ^2}({P_{n}}\,{P_{m}} + {P_{n+m+3}}) \\
 &+ p\,q\,(2\,{P_{n + 1}}\,{P_{m+1}}+ 2\,{P_{n+m}} + {P_{n+m+2}})\\
 &+ {q ^2}\,({P_{n - 1}}\,{P_{m - 1}}+{P_{n+m+1}})\,.
\end{array}
\end{aligned}
\end{equation*}
\\
Then, from the equations (\ref{2-31}), (\ref{2-32}) and (\ref{2-33}), we have the equation (\ref{2-30}).
\end{proof}
\noindent\textbf{Special Case-1:} For the dot product of the generalized Pell vectors $\overrightarrow {{\mathbb{P}_{n}}}$ and $\overrightarrow {{\mathbb{P}_{n+1}}}$, we get

\noindent\begin{equation} \label{2-35}
\begin{aligned}
\begin{array}{lll}
\left\langle {{\overrightarrow {{\mathbb{P}_{n}}}}}, {{\overrightarrow {{\mathbb{P}_{n+1}}}}} \right\rangle & ={\mathbb{P}_{n}} {\mathbb{P}_{n+1}}+{\mathbb{P}_{n+1}}{\mathbb{P}_{n+2}}+{\mathbb{P}_{n+2}}{\mathbb{P}_{n+3}}\\
\\
& =\ {p ^2}\,({P_{2n+4}} + {P_{n}}\,{P_{n + 1}}) +  p\,q\,[2\,{P_{2n+3}} +\,{P_{2n-1}} + 2\,P_{n}\,P_{n-1} ] \\
&\quad \quad + {q^2}\,({P_{2n + 2}} + {P_{n - 1}}\,{P_{n}}\,)
\end{array}	
\end{aligned}
\end{equation}
and
\\
\noindent\begin{equation} \label{2-36}
\begin{aligned}
\begin{array}{lll}
\left\langle {{\overrightarrow {{\mathbb{P}_{n}}}}}, {{\overrightarrow {{\mathbb{P}_{n}}}}} \right\rangle & = ({\mathbb{P}_{n}})^2 +({\mathbb{P}_{n+1}})^2+({\mathbb{P}_{n+2}})^2\\
&  = {p^2}({P_{2n+3}} + {P_n^2}\,) + 2\,p\,q\,({P_{2n+2}}+{P_{n}}\,{P_{n - 1}})\\
&\quad \quad + {q^2}\,(\, {P_{2n+1}} + {P_{n-1}^2}\,).
\end{array}	
\end{aligned}
\end{equation}
\\
Then for the norm of the generalized Pell vector, we have
\begin{equation}\label{2-37}	
\begin{aligned}
\begin{array}{rl}
\left\| {{\overrightarrow {{\mathbb{P}_{n}}}}} \right\|^2 =& \left\langle \overrightarrow {{\mathbb{P}_{n}}},\overrightarrow {{\mathbb{P}_{n}}} \right\rangle= {\mathbb{P}_{n}}^2 +{\mathbb{P}_{n+1}}^2+{\mathbb{P}_{n+2}}^2 \\
=& {p^2}\left( {{P_{2n + 3}} + P_n^2 } \,\right) \\ 
& + \,2\,p \, q\,\left({P_{2n + 2}}+{P_{n}}\,{P_{n - 1}}\, \right)  \\
& + \,{q^2}\left( \,{{P_{2n + 1}} + P_{n - 1}^2}\,\right) .\\
\end{array}
\end{aligned}
\end{equation} \,	
\begin{theorem}
Let ${\overrightarrow {{\mathbb{P}_{n}}}}$  and ${\overrightarrow {{\mathbb{P}_{m}}}}$  be two generalized Pell vectors. The cross product of ${\overrightarrow {{\mathbb{P}_{n}}}}$  and ${\overrightarrow {{\mathbb{P}_{m}}}}$  is given by

\begin{equation}\label{2-38}
{\overrightarrow {{\mathbb{P}_{n}}}} \times {\overrightarrow {{\mathbb{P}_{m}}}} = ( - 1)^{m+1}\,{P_{n-m}}\,(i + 2\, j - k)\,{e_{p}}.
\end{equation}
\end{theorem}

\begin{proof}
The cross product of ${\overrightarrow {{\mathbb{P}_{n}}}} \times {\overrightarrow {{\mathbb{P}_{m}}}}$ defined by \\ 
\\
\begin{equation}\label{2-39} 
\begin{array}{rl}
{\overrightarrow {{\mathbb{P}_{n}}}} \times {\overrightarrow {{\mathbb{P}_{m}}}}
=&\begin{vmatrix}
	i & j & k \\
	{\mathbb{P}_{n}} & {\mathbb{P}_{n+1}} & {\mathbb{P}_{n+2}}  \\
	{\mathbb{P}_{m}} & {\mathbb{P}_{m+1}} & {\mathbb{P}_{m+2}}  \\
\end{vmatrix} \\ \\
=&\,i\,({\mathbb{P}_{m+2}}{\mathbb{P}_{n+1}}-{\mathbb{P}_{m+1}}{\mathbb{P}_{n+2}}) \\
&-j\,({\mathbb{P}_{m+2}}{\mathbb{P}_{n}}-{\mathbb{P}_{m}}{\mathbb{P}_{n+2}}) \\
&+k\,({\mathbb{P}_{m+1}}{\mathbb{P}_{n}}-{\mathbb{P}_{m}}{\mathbb{P}_{n+1}}) \,.
\end{array}
\end{equation}
\\
Now, we calculate the cross products:
\\ \\
Using relations (\ref{2-1}) and (\ref{2-3}) or the following properties  
\begin{equation}\label{2.40}
\left\{\begin{array}{l}
{{P}_{n}^2}-P_{n+1} P_{n-1}=(-1)^{n+1},\\
{{P}_{n-1}}\,P_{m+1}-P_{n+2}\,P_{m-1} =5\,(-1)^{m-1}\,P_{n-m},\\
{{P}_{n-1}}\,P_{m+1}-P_{n+1}\,P_{m-1} =2\,(-1)^{m-1}\,P_{n-m},\\
{{P}_{n}}\,P_{m+2}-P_{n+2}\,P_{m} =2\,(-1)^{m}\,P_{n-m},\\
\end{array}\right.
\end{equation}\\
\, we get,
\begin{equation}\label{2-41}
{\mathbb{P}_{m+2}}{\mathbb{P}_{n+1}}-{\mathbb{P}_{m+1}}{\mathbb{P}_{n+2}}=\,(-1)^{m+1}\,{P_{n-m}}\,({p^2}-2\,p\,q-{q^2})=(-1)^{m+1}\,e_{p},
\end{equation} \,
\begin{equation}\label{2-42}
{\mathbb{P}_{m+2}}{\mathbb{P}_{n}}-{\mathbb{P}_{m}}{\mathbb{P}_{n+2}}=\,(-1)^{m}\,{P_{n-m}}\,({p^2}-2\,p\,q-{q^2})=(-1)^{m}\,e_{p}, 
\end{equation} \\
and
\begin{equation}\label{2-43}
{\mathbb{P}_{m+1}}{\mathbb{P}_{n}}-{\mathbb{P}_{m}}{\mathbb{P}_{n+1}}=\,(-1)^{m}\,{P_{n-m}}\,({p^2}-2\,p\,q-{q^2})=(-1)^{m}\,e_{p}. 
\end{equation}
\par \,
Then from the equations (\ref{2-41}), (\ref{2-42}) and (\ref{2-43}), we obtain the equation (\ref{2-38}).\\
\end{proof}

\begin{theorem}
Let ${\overrightarrow {{\mathbb{P}_{n}}}}$,\, ${\overrightarrow {{\mathbb{P}_{m}}}}$  and \, ${\overrightarrow {{\mathbb{P}_{l}}}}$  be the generalized Pell vectors. The mixed product of these vectors is

\begin{equation}\label{2-44}
\left\langle {{\overrightarrow {{\mathbb{P}_{n}}}} \times {\overrightarrow {{\mathbb{P}_{m}}}}\;,\;{\overrightarrow {{\mathbb{P}_{l}}}}} \right\rangle  = 0.
\end{equation}
\end{theorem}

\begin{proof}
Using  
$\overrightarrow {{\mathbb{P}_{l}}}=({\mathbb{P}_{l}},\,{\mathbb{P}_{l+1}},\,{\mathbb{P}_{l+2}})$, we can write
\\ \,
\begin{equation}\label{2-45} 
\begin{array}{rl}
\left\langle {{\overrightarrow {{\mathbb{P}_{n}}}} \times {\overrightarrow {{\mathbb{P}_{m}}}}\;,\; \,{\overrightarrow{{\mathbb{P}_{l}}}}} \right\rangle
=&\begin{vmatrix}
	{\mathbb{P}_{n}} & {\mathbb{P}_{n+1}} & {\mathbb{P}_{n+2}} \\
	{\mathbb{P}_{m}} & {\mathbb{P}_{m+1}} & {\mathbb{P}_{m+2}}  \\
	{\mathbb{P}_{l}} & {\mathbb{P}_{l+1}} & {\mathbb{P}_{l+2}}  \\
\end{vmatrix} \\ \\
=&\,{\mathbb{P}_{n}}\,({\mathbb{P}_{m+1}}\,{\mathbb{P}_{l+2}}-{\mathbb{P}_{m+2}}\,{\mathbb{P}_{l+1}}) \\
&+{\mathbb{P}_{n+1}}\,({\mathbb{P}_{m+2}}\,{\mathbb{P}_{l}}-{\mathbb{P}_{m}}\,{\mathbb{P}_{l+2}}) \\
&+{\mathbb{P}_{n+2}}\,({\mathbb{P}_{m}}\,{\mathbb{P}_{l+1}}-{\mathbb{P}_{m+1}}\,{\mathbb{P}_{l}}) \,. \\ \,
\end{array}
\end{equation}
Then using the equations (\ref{2-41}), (\ref{2-42}) and (\ref{2-43}) we obtain

\begin{equation*}\label{2-46}
\begin{array}{rl}
{\mathbb{P}_{n}}\,({\mathbb{P}_{m+1}}\,{\mathbb{P}_{l+2}}-{\mathbb{P}_{m+2}}\,{\mathbb{P}_{l+1}})&+ \, {\mathbb{P}_{n+1}}\,({\mathbb{P}_{m+2}}\,{\mathbb{P}_{l}}-{\mathbb{P}_{m}}\,{\mathbb{P}_{l+2}}) \\ 
&+ \, {\mathbb{P}_{n+2}}\,({\mathbb{P}_{m}}\,{\mathbb{P}_{l+1}}-{\mathbb{P}_{m+1}}\,{\mathbb{P}_{l}}) \\
=& (-1)^{l+1}\,{P_{m-l}}\,({p^2-2\,p\,q-\,q^2})\, \\
&(\,{\mathbb{P}_{n}} + 2\,{\mathbb{P}_{n+1}} - {\mathbb{P}_{n+2}}\,) \\
=& (-1)^{l+1}\,{P_{m-l}}\,\,{e_{p}}\,(\,{\mathbb{P}_{n+2}} -\,{\mathbb{P}_{n+2}}\,) \\
=& 0.  
\end{array}
\end{equation*} 
Thus, we have the equation (\ref{2-44}).
\end{proof}

\section{The Generalized Complex Pell Sequence}
In this section, we define the generalized complex Pell sequence denoted by ${\mathbb{C}_{n}}$.

The generalized complex Pell sequence defined by
\begin{equation}\label{3-1}
{\mathbb{C}_{n}} = {\mathbb{P}_{n }} + i \,{\mathbb{P}_{n + 1}}\;,\,\, i^2=-1 ,\,\, i\ne{0}
\end{equation}
with ${\mathbb{C}_1} = p\,+i\,(2\,p + q) \,,\;{{C}_2} = (2\,p\,+\,q)+i\,(5\,p + 2\,q),\;{\mathbb{C}_3} = (5\,p + 2\,q)\,+\,i\,(12\,p + 5\,q)$,... \,
where $p,q$ are arbitrary integers.
\begin{equation}\label{3-3}
 \begin{aligned}
\begin{array}{lll}
  \mathbb{C}_{n} = (\,p\,- 2\,q +\,i\,q\,)\,{P_{n}} + (\,\,q\, +\,i\,p\,)\,{P_{n + 1}},   \\
\\ 
  \mathbb{C}_{n + 1} = (\,q\,+\,i\,p\,)\,{P_{n}} + [\,p\,+ \,i\,(\,2\,p\,+\,q\,)\,]\,{P_{n + 1}},   \\
\\
  \mathbb{C}_{n + 2} = [\,p\, +\,i\,(\,2\,p\,+\,q\,)\,]\,{P_{n}} + [\,(2\,p\,+\,q\,)+\,i\,(5\,p\,+\,2\,q\,)\,]{P_{n + 1}},
  \end{array}
   \end{aligned}
   \end{equation} 	
\,\,\,\,\,\,\, \vdots \\ \\
where $p, q$ are arbitrary integers. That is, the generalized complex Pell sequence is 
\begin{equation}\label{3-2}
\begin{aligned}
\begin{array}{ll}
p\,+i\,(2\,p\,+\,q\,),\,\,(2\,p +\,q\,) +\,i\,(5\,p + 2\,q\,),\,\,(5\,p + 2\,q\,)\,+\,i\,(12\,p + 5\,q\,),\\  \\
(12\,p + 5\,q)+i\,(29\,p + 12\,q\,),\ldots ,(\,p +\,i\,(2\,p +\, q\,)){P_n} + (\,q +i\,p\,)\,){P_{n - 1}}, \ldots 
\end{array}
\end{aligned}
\end{equation}
\\ 
\textbf{Special Case-1:} From the generalized complex Pell sequence \,($\mathbb{C}_n$) for \,\,$\;p = 1,\;\,q = 0$\, in the equation (\ref{3-2}), we obtain as follows: 
$$ ({C_p})\;:\,\;1 + i\,2 \,,\,\,2 + i\,5 \,,\,\,5 + i\,12 \,,\,\,12 + i\,29 , \ldots \,,(1 + i\,2\,)\,{P_n} + i\,{P_{n - 1}}, \ldots  $$ \\
\textbf{Special Case-2:} From the generalized complex Pell sequence \,($\mathbb{C}_n$) for \,$\;\,p = 2,\;\,q = 2$\, in the equation (\ref{3-2}), we obtain following another special sequence: 
$$ ({C_q})\;:\,\;2 + i\,2 \,,\,\,2 + i\,6 \,,\,\,6 + i\,14 \,,\,\,14 + i\,34 , \ldots \,,(1 + i\,2\,)\,{Q_n} + i\,{Q_{n - 1}}, \ldots  $$ \\
From the  equations (3.1), (3.2) and (3.3), we get the following properties for the generalized complex Pell sequence:
\begin{equation}
\mathbb{C}_{n}^2\, +\, \mathbb{C}_{n + 1}^2 = [\,(\,2\,p - 2\,q\,) + i\,( 2\,p + 2\,\,q\,)\,]\,\mathbb{C}_{2\,n + 1}\,- (\,2\, + 2\,i\,)\,{e_{P}}\,{P_{2\,n + 1}}\,\,,
\end{equation} \,
\begin{equation}
\mathbb{C}_{n + 1}^2\, -\, \mathbb{C}_{n - 1}^2 = 2\,[\,(\,2\,p - 2\, q\,) + i\,( 2\,p + 2\,\,q\,)\,]\,\mathbb{C}_{2\,n}\,- 2\,(\,2\, + 2\,i\,)\,{e_{P}}\,{P_{2\,n}}\,\,,
\end{equation} \,
\begin{equation}
\mathbb{C}_{n + 1}^2\, + \,(2 + 2\,i)\,{e_{P}}\,{P_{n}^2} = [\,(\,1 + 2\,i\,)\,\mathbb{C}_{2\,n + 1}\,\,,\;\;\;\;\;\;\;\;
\end{equation} \,
\begin{equation}
\mathbb{C}_{n - 1}\,\mathbb{C}_{n + 1} - \mathbb{C}_{n}^2 = {( - 1)^n}\,(2 + 2\,i\,)\,\,e_{P}\,\,, \;\;\;\;\;\;\;\;\;\;\;\;\;\;\;\;
\end{equation} \,
\begin{equation}
\mathbb{C}_{m}\,\mathbb{C}_{n + 1} + \mathbb{C}_{m - 1}\,\mathbb{C}_{n} = (2\,i - 2)\,[(2\,p + 2\,q\,)R_{m+n}\,+\,e_{P}\,{P}_{m+n-1}]\,\,,
\end{equation} \, 
\begin{equation}
\mathbb{C}_{n}\,\mathbb{C}_{n + r + 1} -\,\mathbb{C}_{n - s}\,\mathbb{C}_{n + r + s + 1} = {( - 1)^{n + s}\,(2 + 2\,i)\, e_{p}\,{P_{s}\,{P_{r + s + 1}}}}\,,\;\;\;\;
\end{equation} \,
\begin{equation}
\mathbb{C}_{m}\,\mathbb{C}_{n+1} - \mathbb{C}_{m+1}\,\mathbb{C}_{n} = (-1)^n\,{P_{m-n}}\,(2 + 2\,i\,)\,{e_{P}}\,\,,\;\;\;\;\;\;\;\;
\end{equation} \,
\begin{equation}
\quad \quad \mathbb{C}_{n + 1 - r}\,\mathbb{C}_{n + 1 + r} - \mathbb{C}_{n + 1}^2 = (-1)^{n - r}\,(2 + 2\,i\,)\,\,e_{P}\,{P}_{r}^2\,\,,
\;\;\;\;\;\;\;\;\;\;\;\;\;\;\;\;\;\;\;\;\;\;\;
\end{equation} \,
\begin{equation}
\frac{\mathbb{C}_{n + r} + {{( - 1)}^r}\mathbb{C}_{n - r}} {{\mathbb{C}_{n}}} = {Q_{r}}\,,
\end{equation} \,
\\
where ${{e}_{p}} = {p}^{2} - 2\, p\,q -\,{q^2} $.\\
\subsection{The Generalized Complex Pell Vectors}
 A generalized dual Pell vector is defined by
$$
\overrightarrow {{{\rm\mathbb{C}}_{\rm{n}}}} = (\mathbb{C}_n\,,\,\mathbb{C}_{n + 1},\,\mathbb{C}_{n + 2}).
$$
From  equations (\ref{3-1}) and (\ref{3-2}), it can be expressed as
\begin{equation}\label{3-13}
\begin{aligned}
\begin{array}{lll}
\overrightarrow {{{\rm\mathbb{C}}_{\rm{n}}}}&=\overrightarrow {{\mathbb{P}_{n}}}+\,i\,\overrightarrow {{\mathbb{P}_{n+1}}}\\
&=(p - 2\,q + \,i\,q)\overrightarrow {{\rm{P}}}_{n}+ (q + \,i\,p)\overrightarrow {{\rm{P}}}_{\rm{n+1}},
\end{array}
\end{aligned}
\end{equation}
where $\overrightarrow {{\rm\mathbb{C}}}_{\rm{n}}= (\,\rm\mathbb{C}_{n}\,,\,\rm\mathbb{C}_{n + 1},\,\rm\mathbb{C}_{n + 2})$ and $\overrightarrow {{\rm{P}}}_{\rm{n}}= (\,{P_n}\,,\,{P_{n + 1}},\,{P_{n + 2}})$ are the generalized complex Pell vector and the Pell vector, respectively. \\
\\
The product of $\overrightarrow {{{\rm\mathbb{C}}_{\rm{n}}}}$ and $ \lambda \in \mathbb{R}$ is given by 
$$ 	\lambda \,\overrightarrow {{{\rm\mathbb{C}}_{\rm{n}}}} =\lambda \, \overrightarrow {{\mathbb{P}_{n}}}+\,i\, \lambda\, \overrightarrow {{\mathbb{P}_{n+1}}} $$
and, \,
$\overrightarrow {{{\rm\mathbb{C}}_{\rm{n}}}}$ and $\overrightarrow {{{\rm\mathbb{C}}_{\rm{m}}}}$ are equal if and only if 
$$
\begin{aligned}
\begin{array}{lcl}
{\mathbb{P}_{n}}&=&{\mathbb{P}_{m}}, \\
{\mathbb{P}_{n + 1}} &= &{\mathbb{P}_{m + 1}}, \\
{\mathbb{P}_{n + 2}}&=&{\mathbb{P}_{m+2}}\,.
\end{array}
\end{aligned}
$$
\begin{theorem}
Let $\overrightarrow {{{\rm\mathbb{C}}_{\rm{n}}}}$ and $\overrightarrow {{{\rm\mathbb{C}}_{\rm{m}}}}$ be two generalized complex Pell vectors. The dot product of $\overrightarrow {{{\rm\mathbb{C}}_{\rm{n}}}}$ and $\overrightarrow {{{\rm\mathbb{C}}_{\rm{m}}}}$ is given by
\\
\begin{equation} \label{3-14}
\begin{aligned}
\begin{array} {lll}
\left\langle \overrightarrow {{{\rm\mathbb{C}}_{\rm{n}}}},\overrightarrow {{{\rm\mathbb{C}}_{\rm{m}}}} \right\rangle = {p^2}\,( {P_{n+m+1}} + {P_{n+m+3}}+{P_{n+m+5}} )\\
\quad\quad \quad \quad \quad \quad + 2\,p\,q\,({P_{n+m}} +{P_{n+m+2}}+{P_{n+m+4}} ) \\
\quad\quad \quad \quad \quad \quad + {q ^2}\,({P_{n+m-1}} +{P_{n+m+1}}+{P_{n+m+3}} ) \\
\quad\quad \quad \quad \quad \quad +(-1)^n\,i\,\,{P_{m-n}}\,\,e_{P} \\
\quad\quad \quad \quad \quad  = 7\,[\,{p^2}\,{P_{n+m+3}}+2\,p\,q\,{P_{n+m+2}}+\,{q ^2}\,{P_{n+m+1}}\,] \\
\quad\quad \quad \quad \quad \quad +(-1)^n\,i\,\,{P_{m-n}}\,\,e_{P}.
\end{array}
\end{aligned}
\end{equation}
\end{theorem}	

\begin{proof}
The dot product of $\;$
$\overrightarrow {{{\rm\mathbb{C}}_{\rm{n}}}} = (\mathbb{C}_n\,,\,\mathbb{C}_{n + 1},\,\mathbb{C}_{n + 2})$ and \, $\overrightarrow {{{\rm\mathbb{C}}_{\rm{m}}}} = (\mathbb{C}_m\,,\,\mathbb{C}_{m + 1},\,\mathbb{C}_{m+ 2})$ defined by 
$$\begin{aligned}
\begin{array} {lll}
\left\langle \overrightarrow {{{\rm\mathbb{C}}_{\rm{n}}}},\overrightarrow {{{\rm\mathbb{C}}_{\rm{m}}}} \right\rangle & = \mathbb{C}_n \bar{\mathbb{C}}_m+\mathbb{C}_{n+1} \bar{\mathbb{C}}_{m+1}+\mathbb{C}_{n+2} \bar{\mathbb{C}}_{m+2}.\\
\end{array}
\end{aligned}$$	
From the equations (\ref{3-1}), (\ref{3-2}) and (\ref{3-3}), we obtain
\begin{equation} \label{3-15}
\begin{aligned}
\begin{array} {lll}
{{{\rm\mathbb{C}}}_{\rm{n}}} {{\bar{\rm\mathbb{C}}}_{\rm{m}}} = {p^2}( {P_{n+m+1}} )+2\, p\,q\,({P_{n+m}} ) \\
\quad\quad \quad \quad \quad \quad + {q ^2}\,({P_{n+m-1}})+(-1)^n\,i\,{P_{m-n}}\,\,e_{P},
\end{array}
\end{aligned}
\end{equation}
\begin{equation} \label{3-16}
\begin{aligned}
\begin{array} {lll}
{{{\rm\mathbb{C}}}_{\rm{n+1}}}{{\bar{\rm\mathbb{C}}}_{\rm{m+1}}} = {p^2}( {P_{n+m+3}} )+2\,p\,q\,(2\,{P_{n+m+2}} ) \\
\quad\quad \quad \quad \quad \quad \quad  + {q^2} ( {P_{n+m+1}} )+(-1)^{n+1}\,i\,{P_{m-n}}\,\,e_{P},\,
\end{array}
\end{aligned}
\end{equation}
and
\begin{equation} \label{3-17}
\begin{aligned}
\begin{array} {lll}
{{{\rm\mathbb{C}}}_{\rm{n+2}}}{{\bar{\rm\mathbb{C}}}_{\rm{m+2}}} = {p^2}( {P_{n+m+5}} )+2\,p\,q\,({P_{n+m+4}}) \\
\quad\quad \quad \quad \quad \quad \quad  + {q ^2}\,({P_{n+ m+3}}) +(-1)^{n}\,i\,{P_{m-n}}\,\,e_{P}.
\end{array}
\end{aligned}
\end{equation} \\
Then from equations (\ref{3-15}), (\ref{3-16}) and (\ref{3-17}), we have the equation (\ref{3-14}).
\end{proof}
\noindent\textbf{Special Case-3:} For the dot product of generalized complex Pell vectors $\overrightarrow {{{\rm\mathbb{C}}_{\rm{n}}}}$ and $\overrightarrow {{\rm\mathbb{C}}}_{\rm{n+1}}$, we get
\begin{equation} \label{3-18}
\begin{aligned}
\begin{array}{lll}
\left\langle \overrightarrow {{{\rm\mathbb{C}}_{\rm{n}}}}, \overrightarrow {{{\rm\mathbb{C}}_{\rm{n+1}}}} \right\rangle&= \mathbb{C}_n \bar{\mathbb{C}}_{n+1}+\mathbb{C}_{n+1}\bar{\mathbb{C}}_{n+2}+\mathbb{C}_{n+2}\bar{\mathbb{C}}_{n+3}\\
& = {p^2}\,({P_{2n + 2}}+{P_{2n + 4}}+{P_{2n + 6}} ) \\
&\quad+ 2\,p\,q\, ({P_{2n+1}}+{P_{2n + 3}}+{P_{2n + 5}} ) \\
&\quad+ {q^2} ({P_{2n}}+{P_{2n + 2}}+{P_{2n + 4}} )\\
&\quad+\,i\,\{ (-1)^{n}\,e_{P} \} \\
& = 7\,[\,p^2\,{P_{2n+4}}+2\,p\,q\,{P_{2n+3}}+\,q^2\,{P_{2n+2}}\,] \\
&\quad+(-1)^n\,i\,\,e_{P}
\end{array}	
\end{aligned}
\end{equation}

and \\
\begin{equation}\label{3-19}
\begin{aligned}
\begin{array}{lll}
\left\langle {{\overrightarrow {{\rm\mathbb{C}}}_{\rm{n}}}}, {{\overrightarrow {{\rm\mathbb{C}}}_{\rm{n}}}} \right\rangle  & = \mathbb{C}_n \bar{\mathbb{C}}_{n}+\mathbb{C}_{n+1}\bar{\mathbb{C}}_{n+1}+\mathbb{C}_{n+2}\bar{\mathbb{C}}_{n+2}\\
& =|{\mathbb{C}_n}|^2 +|{\mathbb{C}_{n+1}|^2}+{|\mathbb{C}_{n+2}}|^2\\
&  =  {p^2}({P_{2n + 1}}+{P_{2n + 3}}+{P_{2n + 5}}) + 2p\,q\,({P_{2n}}+{P_{2n + 2}}+{P_{2n + 4}}) \\
&\quad + {q^2}({P_{2n - 1}}+{P_{2n + 1}}+{P_{2n + 3}}) \\
&= 7\,[\,{p^2}\,{P_{2n+3}}+2\,p\,q\,{P_{2n+2}}+\,{q ^2}\,{P_{2n+1}}\,] .\\
\end{array}	
\end{aligned}
\end{equation}
\\
Then for the norm of the generalized complex Pell vector, we have \\ \,
\\
\begin{equation}\label{3-20}	
\begin{aligned}
\begin{array}{lll}
\left\| {{\overrightarrow {{{\rm\mathbb{C}}_{\rm{n}}}}}} \right\|& = \sqrt[]{\left[\left\langle \overrightarrow {{{\rm\mathbb{C}}_{\rm{n}}}},\overrightarrow {{{\rm\mathbb{C}}_{\rm{n}}}} \right\rangle\right]}= \sqrt[]{{|{\mathbb{C}_n}|^2} +|{\mathbb{C}_{n+1}|^2}+{|\mathbb{C}_{n+2}}|^2} \\
&=\sqrt[]{ 7\,[\,{{p^2}{P_{2n + 3}} + 2\,p\,q\,{P_{2n + 2}} + {q^2}\,{P_{2n + 1}}\,] }}.
\end{array}
\end{aligned}
\end{equation}
\\	
\textbf{Special Case-4:} For $p=1,\, q=0$, in the equations (\ref{3-14}), (\ref{3-18}) and (\ref{3-20}), we have \\

$ \left\langle {{{\overrightarrow C }_n},{{\overrightarrow C }_m}} \right\rangle = 7\,({P_{n + m + 1}})\, + \,i\, (-1)^n \,{P_{m-n}}\,, $

$ \left\langle {{{\overrightarrow C }_n},{{\overrightarrow C }_{n + 1}}} \right\rangle = 7\,({P_{2n + 4}})+\,i\,(-1)^{n} $\\

and \\

$\left\|{{{\overrightarrow C }_n}} \right\| = {\sqrt[]{  {7\,({P_{2n + 3}}) }}}.$ \\
\begin{theorem}
Let ${\overrightarrow {{{\rm\mathbb{C}}_{\rm{n}}}}}$  and ${\overrightarrow {{{\rm\mathbb{C}}_{\rm{m}}}}}$  be two generalized complex Pell vectors. The cross product of ${\overrightarrow {{{\rm\mathbb{C}}_{\rm{n}}}}}$  and ${\overrightarrow {{{\rm\mathbb{C}}_{\rm{m}}}}}$  is given by
\\
\begin{equation}\label{3-21}
{\overrightarrow {{{\rm\mathbb{C}}_{\rm{n}}}}} \times {\overrightarrow {{{\rm\mathbb{C}}_{\rm{m}}}}} = {( - 1)^{n}}{P_{m - n}}\,(2 + 2\,i )\,e_{P}\,(( i + 2\,j - k).
\end{equation}
\end{theorem}
\begin{proof}
The cross product of $\overrightarrow {{{\rm\mathbb{C}}_{\rm{n}}}}=\overrightarrow {{\mathbb{P}_{n}}}+\,i\,\overrightarrow {{\mathbb{P}_{n+1}}}$ and  $\overrightarrow {{{\rm\mathbb{C}}_{\rm{m}}}}=\overrightarrow {{\mathbb{P}_{m}}}+\,i\,\overrightarrow {{\mathbb{P}_{m+1}}}$ defined by
 
\begin{equation}\label{3-22} 
\begin{array}{rl}
{\overrightarrow {{{\rm\mathbb{C}}_{\rm{n}}}}} \times {\overrightarrow {{{\rm\mathbb{C}}_{\rm{m}}}}} = 
&\begin{vmatrix}
	i & j & k \\
	{{\rm\mathbb{C}}_{n}} & {{\rm\mathbb{C}}_{n+1}} & {{\rm\mathbb{C}}_{n+2}}  \\
	{{\rm\mathbb{C}}_{m}} & {{\rm\mathbb{C}}_{m+1}} & {{\rm\mathbb{C}}_{m+2}}  \\
\end{vmatrix} \\ \\
=&\,i\,({{\rm\mathbb{C}}_{m+2}}{{\rm\mathbb{C}}_{n+1}}-{{\rm\mathbb{C}}_{m+1}}{{\rm\mathbb{C}}_{n+2}}) \\
&-j\,({{\rm\mathbb{C}}_{m+2}}{{\rm\mathbb{C}}_{n}}-{{\rm\mathbb{C}}_{m}}{{\rm\mathbb{C}}_{n+2}}) \\
&+k\,({{\rm\mathbb{C}}_{m+1}}{{\rm\mathbb{C}}_{n}}-{{\rm\mathbb{C}}_{m}}{{\rm\mathbb{C}}_{n+1}}) \,.
\end{array}
\end{equation} \,
\\
Using the property $P_m P_{n+1}-P_{m+1}P_n=(-1)^n P_{m-n}$, we get 

\begin{equation}\label{3-23}
{{\rm\mathbb{C}}_{m+2}}{{\rm\mathbb{C}}_{n+1}}-{{\rm\mathbb{C}}_{m+1}}{{\rm\mathbb{C}}_{n+2}}=2\,(-1)^{n}\,(1+i)\,{P_{m-n}}\,({p^2}-2\,p\,q-{q^2}),
\end{equation} \,
\begin{equation}\label{3-24}
{{\rm\mathbb{C}}_{m+2}}{{\rm\mathbb{C}}_{n}}-{{\rm\mathbb{C}}_{m}}{{\rm\mathbb{C}}_{n+2}}=4\,(-1)^{n}\,(1+i)\,{P_{m-n}}\,({p^2}-2\,p\,q-{q^2})
\end{equation} \\
and \\
\begin{equation}\label{3-25}
{{\rm\mathbb{C}}_{m+1}}{{\rm\mathbb{C}}_{n}}-{{\rm\mathbb{C}}_{m}}{{\rm\mathbb{C}}_{n+1}}=2\,(-1)^{n+1}\,(1+i)\,{P_{n-m}}\,({p^2}-2\,p\,q-{q^2}). 
\end{equation}
\par \,
Then from the equations (\ref{3-23}), (\ref{3-24}) and (\ref{3-25}), we obtain the equation (\ref{3-21}).\\ \\ \,
\textbf{Special Case-5:} For $p=1,\,q=0$ in the equation (\ref{3-19}), we have 
$${\overrightarrow {{{\rm\mathbb{C}}_{\rm{n}}}}} \times {\overrightarrow {{{\rm\mathbb{C}}_{\rm{m}}}}}=(-1)^{n} P_{m-n}\,(2+2\,i) (i+2j-k). $$
\end{proof}
\begin{theorem}
Let ${\overrightarrow {{{\rm\mathbb{C}}_{\rm{n}}}}}$, ${\overrightarrow {{{\rm\mathbb{C}}_{\rm{m}}}}}$  and ${\overrightarrow {{{\rm\mathbb{C}}_{\rm{k}}}}}$  be the generalized complex Pell vectors. The mixed product of these vectors is
\begin{equation}\label{3-26}
\left\langle {{\overrightarrow {{{\rm\mathbb{C}}_{\rm{n}}}}} \times {\overrightarrow {{{\rm\mathbb{C}}_{\rm{m}}}}}\;,\;{\overrightarrow {{{\rm\mathbb{C}}_{\rm{k}}}}}} \right\rangle  = 0 .
\end{equation}
\end{theorem}
\begin{proof}
Using \,\,
$\overrightarrow {{{\rm\mathbb{C}}_{\rm{k}}}}=\overrightarrow {{\mathbb{P}_{k}}}+\,i\overrightarrow {{\mathbb{P}_{k+1}}}$,\, we can write\\ 
\begin{equation}\label{3-27} 
\begin{array}{rl}
\left\langle {{\overrightarrow {{{\rm\mathbb{C}}_{\rm{n}}}}} \times {\overrightarrow {{{\rm\mathbb{C}}_{\rm{m}}}}}\;,\; \,{\overrightarrow{{{\rm\mathbb{C}}_{\rm{k}}}}}} \right\rangle 
&=(-1)^n\,P_{m-n}\,e_{P}\,(2+2\,i)\left\langle (i+2j-k),\,\overrightarrow {{{\rm{C}}_{\rm{k}}}} \right\rangle. 
\end{array}
\end{equation}
\\
So we have 
$$
\left\langle (i+2j-k),\,\overrightarrow {{{\rm{C}}_{\rm{k}}}} \right\rangle= {C_{k}}+2\,{C_{k+1}}-{C_{k+2}}=0,
$$
$$\left\langle (i+2j-k),\,\overrightarrow {{{\rm{C}}_{\rm{k+1}}}} \right\rangle= {C_{k+1}}+2\,{C_{k+2}}-{C_{k+3}}=0 .$$
\\
Thus, we have the equation (\ref{3-26}).
\end{proof}

\section{The Generalized Dual Pell Sequence } 

\par In this section, we define the generalized dual Pell sequence denoted by  $(\mathbb{D}^P_n).$ \,
\par \,
The n-th term of the generalized dual Pell number defined by
\begin{equation}\label{4-1}
{\mathbb{D}^P_{n}} = {\mathbb{P}_{n}}+ \varepsilon\,{\mathbb{P}_{n+1}}.
\end{equation}
Using equations (\ref{2-1}), (\ref{2-2}) and (\ref{2-3}), we get
\begin{equation}\label{4-2}
{\mathbb{D}^P_{n}} = (\,p\,+\,\varepsilon\,(\,2\,p + q\,))\,{P_{n}} + (\,q\,+\,\varepsilon\,p\,){P}_{n - 1}\;,\;\;(n \ge 2)
\end{equation}
with ${\mathbb{D}^P_{0}} = q + \varepsilon\,p,\,\;{\mathbb{D}^P_{1}} = p + \varepsilon\,(2\,p\,+\,q\,),\,\;{\mathbb{D}^P_{2}} = (2\,p\,+\,q\,) + \,\varepsilon(5\,p\,+\,2\,q\,)$ \, where $\,{\varepsilon ^2} = 0,\,\,\,\varepsilon  \ne 0$ \,\cite{ercan}. \\ \\ 
That is, the generalized dual Pell sequence is 
\begin{equation}\label{4-3}
\begin{aligned}
\begin{array}{ll}
({\mathbb{D}^P_{n}}):\,\,\,q + \varepsilon\,p\,,\,\,p\,+\,\varepsilon\,(2p\,+\,q\,)\,,\,\,(2p\,+\,q\,) + \varepsilon\,(5p\,+\,2\,q\,), \\
\\
(5p\,+\,2q\,) + \varepsilon\,(\,12p\,+\,5q\,), \,\ldots\,,(\,p\,+\,\varepsilon\,(\,2\,p + q\,))\,{P_{n}} + (\,q\,+\,\varepsilon\,p\,){P}_{n - 1}\, \ldots
\end{array}
\end{aligned}
\end{equation}
Using the equation (\ref{4-1}) and (\ref{4-2}), we write

\begin{equation}\label{4-4}
 \begin{aligned}
\begin{array}{lll}
  {\mathbb{D}^P_{n}} = (\,p\,- 2\,q +\,\varepsilon\,\,q\,)\,{P_{n}} + (\,\,q\, + \varepsilon\,\,p\,)\,{P_{n + 1}},   \\
\\ 
  {\mathbb{D}^P_{n + 1}} = (\,q\,+\,\varepsilon\,\,p\,)\,{P_{n}} + [\,\,p\, + \varepsilon\,(2\,p\,+\,q\,)\,]\,{P_{n + 1}},   \\
\\
  {\mathbb{D}^P_{n + 2}} = [\,p\, + \varepsilon\,(\,2\,p\,+\,q\,)\,]\,{P_{n}} + [\,(2\,p\,+\,q\,)+\,\varepsilon\,(5\,p\,+\,2\,q\,)\,]{P_{n + 1}},
  \end{array}
   \end{aligned}
   \end{equation} 	
\,\,\,\,\,\,\, \vdots \\ \\
From the  equations (\ref{4-1}), (\ref{4-2}), (\ref{4-3}) and (\ref{4-4}), we get the following properties for the generalized dual Pell sequence:
\begin{equation}
({\mathbb{D}^P_{n}})^2 +\,({\mathbb{D}^P_{n-1}})^2  = [ (2\,p - 2\,q) +\varepsilon (2\,p + 2\,q)\,]\,{\mathbb{D}^P_{2n - 1}} - \,e_{P}\,(1 + 2\,\varepsilon )\,{P_{2n - 1}}\,,\;\;\;\;
\end{equation}
\begin{equation}
({\mathbb{D}^P_{n + 1}})^2 +\,({\mathbb{D}^P_n})^2   = [ (2\,p - 2\,q) +\varepsilon (2\,p + 2\,q)\,]\,{\mathbb{D}^P_{2n + 1}} - \,e_{P}\,(1 + 2\,\varepsilon )\,{P_{2n + 1}}\,,\;\;\;\;
\end{equation}
\begin{equation}
({\mathbb{D}^P_{n + 1}})^2 -\,({\mathbb{D}^P_{n - 1}})^2   = [ (4\,p - 4\,q) +\varepsilon (4\,p + 4\,q)\,]\,{\mathbb{D}^P_{2n}} - 2\,e_{P}\,(1 + 2\,\varepsilon )\,{P_{2n}}\,,\;\;\;\;
\end{equation}
\begin{equation}
{\mathbb{D}^P_{n - 1}}\,{\mathbb{D}^P_{n + 1}} - ({\mathbb{D}^P_n})^2 = {( - 1)^n}\, e_{P}\,(1 + 2\,\varepsilon )\,\,,\;\;\;\;
\end{equation}
\begin{equation}
{\mathbb{D}^P_{m}}\,{\mathbb{D}^P_{n + 1}} - {\mathbb{D}^P_{m+1}}\,{\mathbb{D}^P_{n}} = {( - 1)^n}\, e_{P}\,(1 + 2\,\varepsilon )\,P_{m-n}\,,\;\;\;\;
\end{equation}
\begin{equation}
({\mathbb{D}^P_{n + 1}})^2 + e_{P}\,(1 + 2\,\varepsilon )\,\,{P}_{n}^2 = \left( 1 + 2\,\varepsilon \,\right)\,{\mathbb{D}^P_{2n + 1}}\,,\;\;\;\;
\end{equation}
\begin{equation}
{\mathbb{D}^P_n}\,{\mathbb{D}^P_{n + r + 1}} - {\mathbb{D}^P_{n - s}}\,{\mathbb{D}_{n + r + s + 1}} = {( - 1)^{n + s}}\,e_{P}\,\,(1 + 2\, \varepsilon )\,{P_s}{P_{r + s + 1}}\,\,,\;\;\;\;
\end{equation}
\begin{equation}
{\mathbb{D}^P_{n + 1 - r}}\,{\mathbb{D}^P_{n + 1 + r}} - ({\mathbb{D}^P_{n + 1}})^2 = {( - 1)^{n - r}}\,e_{P}\,(1 + 2\,\varepsilon )\,\,{P}_{r}^2\,\,,\;\;\;\;
\end{equation}
\begin{equation}
\frac{{\mathbb{D}^P_{n + r}} + ( - 1)^r\,{\mathbb{D}^P_{n - r}}}  {{\mathbb{D}^P_n}} = {Q_r},
\end{equation}
where ${{e}_\mathbb{D}} = {e}_{P}\,(1 + 2\,\varepsilon).$ \\

\noindent\textbf{Special Case-1:} From the generalized dual Pell sequence \,($\mathbb{D}^P_n$)\, for $\;\,p = 1,\,\;q = 0$ in the equation (\ref{4-3}), we obtain dual Pell sequence ($D^P_n$) as follows: 
$$ ({D^P_n})\;:\,\;1 + \varepsilon\,2 \,,\,\,2 + \varepsilon\,5 \,,\,\,5 + \varepsilon\,12 ,\,\,12 + \varepsilon\,29 , \ldots \,,(1 + \varepsilon\,2 )\,{P_{n}} + \varepsilon \,{P_{n - 1}}, \ldots  $$ 
\begin{theorem}
	If \,${\mathbb{P}_{n}}$ and \,${\mathbb{D}^P_n}$ are the generalized Pell number and   generalized dual Pell number respectively, then
	$$\mathop {\lim }\limits_{n \to \infty } \frac{{\mathbb{P}_{n+1}}} {{\mathbb{P}_{n}}} = \frac{p\alpha+q} {\alpha\,q + ( p - 2\,q )}$$
	and
	$$\mathop {\lim }\limits_{n \to \infty } \frac{{\mathbb{D}^P_{n + 1}}} {{\mathbb{D}^P_n}} = \frac{(p\,q)\,\alpha^2 + (p^2 - 2p\,q + q^2)\,\alpha + (p\,q - 2q^2)}  {\alpha^2\,( q^2 ) + \alpha\,( 2p\,q - 4\,q^2 ) + (p^2 - 4p\,q + 4q^2)}.$$		
\end{theorem} \, \,
\begin{proof}
	
	For the Pell number $P_{n}$, we had
	$$\mathop {\lim }\limits_{n \to \infty } \frac{{P_{n + 1}}} {{P_n}} = \alpha, $$
	where $\alpha  = {1 + \sqrt 2 }.$ \\ 
	
	Then for the generalized dual Pell number $\mathbb{D}^P_n$,  we obtain

\begin{equation*}
\begin{aligned} 
\begin{array}{lll}
\mathop {\lim }\limits_{n \to \infty } \frac{{\mathbb{D}^P_{n + 1}}} {{\mathbb{D}^P_n}}& = \mathop {\lim }\limits_{n \to \infty } \frac{[p + \varepsilon\,(2p+q)]{P_{n + 1}} + (q + \varepsilon\,p){P_{n}}}  {{[p + \varepsilon \,(2p+q)]{P}_{n}} + (q + \varepsilon\,p)\,{P_{n-1}}} \\
\\

&  = \mathop {\lim }\limits_{n \to \infty } \frac{{[\,(2\,p\,+ q)\,P_{n}\,+p\,P_{n - 1}] + \varepsilon\,[\,(5p\,+2q)\,P_{n} + (2p\,+q)\,P_{n - 1}] }} {(p\,P_{n} + q\,P_{n-1})\,+\varepsilon\,[\,(2p+q)\,P_{n} + p\,P_{n - 1}]}\, \\ \\

&  = \mathop {\lim }\limits_{n \to \infty } \frac{{(2p^2+p\,q)\,P_{n}^2 + (p^2 + 2p\,q + q^2)\,P_{n}P_{n-1} + p\,q\,P_{n-1}^2 \,}} {p^2\,P_{n}^2 + 2\,p\,q\,P_{n}P_{n-1} +\,q^2\,P_{n-1}^2 }
\\
\\

&\;\;\;\;\;+\mathop{\lim}\limits_{n \to \infty } \varepsilon\,\,\frac{{{{ (p^2} - 2\,p\,q - q^2)\,(P_{n}^2}}-2P_{n}P_{n-1} -\,P_{n-1}^2 )} {{p^2\,P_{n}^2 + (2p\,q\,)\,P_{n}P_{n-1} +\,q^2\,P_{n-1}^2}}\\
\\
\end{array}
\end{aligned}
\end{equation*}
\begin{equation} \label{4-14}
\begin{aligned} 
\begin{array}{lll}

& = \frac{(2p^2+p\,q)\,\alpha^2 + (p^2 + 2p\,q + q^2)\,\alpha + p\,q} {{p^2\,\alpha^2 + (2p\,q\,)\,\alpha +\,q^2\,}}
\\
\\
&\;\;\;\;\;+ \varepsilon \,\frac{\,\,{{{ (p^2} - 2\,p\,q - q^2)\,(\alpha^2}}-2\alpha -1 )} {{p^2\,\alpha^2 + (2p\,q\,)\,\alpha +\,q^2\,}}
\\
\\
& = \frac{(2p^2+p\,q)\,\alpha^2 + (p^2 + 2p\,q + q^2)\,\alpha + p\,q} {{p^2\,\alpha^2 + (2p\,q\,)\,\alpha +\,q^2\,}}+0.	
\end{array}
\end{aligned}
\end{equation}
	\end{proof}
\noindent\textbf{Special Case-2:} If we take $p = 1$, $q = 0$ in the equation (\ref{4-14}), we obtain
	
	$$\mathop {\lim }\limits_{n \to \infty } \frac{{\mathbb{D}^P_{n + 1}}} {{\mathbb{D}^P_n}}=\mathop {\lim }\limits_{n \to \infty } \frac{{D^P_{n + 1}}} {{D^P_n}} = \alpha  + 0 = \alpha. $$
\\
\begin{theorem}
	The Binet's formula \footnote[2]{Binet's formulas are the explicit formulas to obtain the n-th Pell and Pell-Lucas numbers. It is well known that Binet's formulas for the Pell and Pell-Lucas numbers are \\
		$$
		{P_n} = \frac{{\alpha ^n} - {\beta ^n}} {\alpha  - \beta }
		$$
		and  
		$$
		{Q_n} = {\alpha ^n} + {\beta ^n}
		$$
		respectively, where $\alpha  + \beta  = 2\,,\;\alpha  - \beta  = 2\sqrt {2\,} \,,\;\alpha \beta  =  - 1$    and $\alpha  = {{1 + \sqrt 2 } }$ , \,$\beta  = {{1 - \sqrt 2 } }$ , \,\cite{bicknell, horadam5}.} 
	for  the generalized dual Pell sequence is as follows;		
\begin{equation}
	{\mathbb{D}^P_n} = \frac {1} {\alpha  - \beta }\,(\,\bar \alpha \,\,{\alpha ^n} - \bar \beta \,\,{\beta ^n}).\;
\end{equation}
\end{theorem}
\begin{proof}
	If we use definition of the generalized dual Pell sequence and  substitute first equation in footnote, we get 	
	\begin{equation*}
	\begin{aligned}
	\begin{array}{lll}
	{\mathbb{D}^P_n} &= (p - 2\,q + \varepsilon \,q)\,{P_n} + (q + \varepsilon \,p)\,{P_{n + 1}}\\
	& = (p - 2\,q + \varepsilon \,q)\,(\frac{{\alpha ^n} - {\beta ^n}} {\alpha  - \beta }) + (q + \varepsilon \,p)\,(\frac{{\alpha ^{n + 1}} - {\beta ^{n + 1}}} {\alpha  - \beta }) \\ 
	& = \frac{\overline \alpha  \;{\alpha ^n}\; - \,\overline {\beta \,} \,{\beta ^n}} {\alpha  - \beta },\,
	
	\end{array}
	\end{aligned}
	\end{equation*}
	where  $\overline \alpha   = (p - 2\,q + \varepsilon \,q) + \alpha \,(q + \varepsilon \,p)$ and  $\overline \beta   = (p - 2\,q + \varepsilon \,q) + \beta \,(q + \varepsilon \,p) .$  
\end{proof}

\subsection{The Generalized Dual Pell Vectors}
 A generalized dual Pell vector is defined by

$$
\overrightarrow {{\mathbb{D}^P_n}} = (\mathbb{D}^P_n\,,\,\mathbb{D}^P_{n + 1},\,\mathbb{D}^P_{n + 2}).
$$
From the equations (\ref{2-1}), (\ref{2-2}) and (\ref{2-3}), it can be expressed as
\begin{equation}\label{4-17}
\begin{aligned}
\begin{array}{lll}
\overrightarrow {{\mathbb{D}^P_n}}&=\overrightarrow {{\mathbb{P}_{n}}}+\varepsilon\,\overrightarrow {{\mathbb{P}_{n+1}}}\\
&=(p - 2\,q + \varepsilon\,q)\overrightarrow {{\rm{P}}}_{n}+ (q + \varepsilon \,p)\overrightarrow {{\rm{P}}}_{\rm{n+1}},
\end{array}
\end{aligned}
\end{equation}
where $\overrightarrow {{\mathbb{P}_{n}}}= (\,{\mathbb{P}_{n}}\,,\,{\mathbb{P}_{n+1}},\,{\mathbb{P}_{n+2}})$ and $\overrightarrow {{\rm{P}}}_{\rm{n}}= (\,{P_n}\,,\,{P_{n + 1}},\,{P_{n + 2}})$ are the generalized Pell vector and the Pell vector, respectively. \\
\\
The product of\, $\overrightarrow {{\mathbb{D}^P_n}}$ and $ \lambda \in \mathbb{R}$ is given by 
$$ 	\lambda \,\overrightarrow {{\mathbb{D}^P_n}} =\lambda \overrightarrow {{\mathbb{P}_{n}}}+\varepsilon \,\lambda \,\overrightarrow {{\mathbb{P}_{n+1}}} $$
and, \,
$\overrightarrow {{\mathbb{D}^P_n}}$ and $\overrightarrow {{\mathbb{D}^P_m}}$ are equal if and only if 
$$
\begin{aligned}
\begin{array}{lcl}
{\mathbb{P}_{n}}&=&{\mathbb{P}_{m}}, \\
{\mathbb{P}_{n + 1}} &= &{\mathbb{P}_{m + 1}}, \\
{\mathbb{P}_{n + 2}}&=&{\mathbb{P}_{m+2}}\,.
\end{array}
\end{aligned}
$$
\begin{theorem}
Let $\overrightarrow {{\mathbb{D}^P_n}}$ and $\overrightarrow {{\mathbb{D}^P_m}}$ be two generalized dual Pell vectors. The dot product of \, $\overrightarrow {{\mathbb{D}^P_n}}$ and $\overrightarrow {{\mathbb{D}^P_m}}$ is given by
\begin{equation} \label{4-18}
\begin{aligned}
\begin{array} {lll}
\left\langle \overrightarrow {{\mathbb{D}^P_n}},\,\overrightarrow {{\mathbb{D}^P_m}} \right\rangle = {p^2}[( {{P_{n+m+3}} + {P_{n}}\,{P_m}} )+ \varepsilon\,(2{P_{n+m+4}} + {P_{n+m}} + 2\,{P_{n}}{P_{m}})\, ]    \\
\quad\quad \quad \quad \quad + p\,q [ ({{P_{n+m+2}} + 2\,{P_{n+1}\,{P_{m+1}} + 2\,{P_{n+m}}}} ) \\
\quad\quad \quad \quad \quad +\varepsilon\,(2{P_{n+m+3}} + 4\,{P_{n+m+1}} + 2\,{P_{n+m+2}} + 4\,{P_{n+1}}{P_{m+1}})\, ] \\
\quad\quad \quad \quad \quad  + {q ^2} [ ({{P_{n+m+1}} + {P_{n-1}}\,{P_{m-1}}}) \\
\quad\quad \quad \quad \quad +\varepsilon\,(2{{P_{n+m+2}} + {P_{n+m-2}}\, + 2\,{P_{n-1}}{P_{m-1}}})\, ].
\end{array}
\end{aligned}
\end{equation}
\end{theorem}	

\begin{proof}
The dot product of $\;$
$\overrightarrow {{\mathbb{D}^P_n}} = ({\mathbb{D}^P_n}\,,\,{\mathbb{D}^P_{n+1}},\,{\mathbb{D}^P_{n + 2}})$ \, and \\
$\overrightarrow {{\mathbb{D}^P_m}} = ({\mathbb{D}^P_m}\,,\,{\mathbb{D}^P_{m+1}},\,{\mathbb{D}^P_{m + 2}})$ is defined by 
$$\begin{aligned}
\begin{array} {lll}
\left\langle \overrightarrow {{\mathbb{D}^P_n}},\,\overrightarrow {{\mathbb{D}^P_m}} \right\rangle & = {\mathbb{D}^P_n}\,{\mathbb{D}^P_m}+{\mathbb{D}^P_{n+1}}\,{\mathbb{D}^P_{m+1}}+{\mathbb{D}^P_{n+2}}\,{\mathbb{D}^P_{m+2}}\\
&=\left\langle \overrightarrow {{\mathbb{P}_{n}}}
,\overrightarrow {{\mathbb{P}_{m}}} \right\rangle+\varepsilon\,[\, \left\langle \overrightarrow {{\mathbb{P}_{n}}}\;,\overrightarrow {{\mathbb{P}_{m+1}}} \right\rangle +\left\langle \overrightarrow {{\mathbb{P}_{n+1}}}
\;,\overrightarrow {{\mathbb{P}_{m}}} \right\rangle \, ],
\end{array}
\end{aligned}$$	
where $\overrightarrow {{\mathbb{P}_{n}}}= (\,{\mathbb{P}_{n}}\,,\,{\mathbb{P}_{n+1}},\,{\mathbb{P}_{n+2}})$ is the generalized Pell vector. Also, from the equations (\ref{2-1}), (\ref{2-2}) and (\ref{2-3}), we obtain
\begin{equation} \label{4-19}
\begin{aligned}
\begin{array} {lll}
\left\langle \overrightarrow {{\mathbb{P}_{n}}}
,\overrightarrow {{\mathbb{P}_{m}}} \right\rangle = {p^2}( {{P_{n+m+3}} + {P_{n}}\,{P_m}} ) \\
\quad\quad \quad \quad \quad \quad + p\,q\,({{P_{n+m+2}} + 2\,{P_{n+1}\,{P_{m+1}} + 2\,{P_{n+m}}}} ) \\
\quad\quad \quad \quad \quad \quad + {q ^2}\,({{P_{n+m+1}} + {P_{n-1}}\,{P_{m-1}}})\,,
\end{array}
\end{aligned}
\end{equation}
\begin{equation} \label{4-20}
\begin{aligned}
\begin{array} {lll}
\left\langle \overrightarrow {{\mathbb{P}_{n}}}\;,\overrightarrow {{\mathbb{P}_{m+1}}} \right\rangle = {p^2}( {P_{n+m+4}} + {P_{n}}\,{P_{m+1}} ) \\
\quad\quad \quad \quad \quad \quad \quad + p\,q\,[2\,{P_{n+m+3}} + 2\,{P_{n+1}}\,{P_{m+2}} + 2\,{P_{n+m+1}} ] \\
\quad\quad \quad \quad \quad \quad \quad  + {q^2} [ {{P_{n+m +2}} + {P_{n - 1}}\,{P_{m}}} ]\,
\end{array}
\end{aligned}
\end{equation}
and
\begin{equation} \label{4-21}
\begin{aligned}
\begin{array} {lll}
\left\langle \overrightarrow {{\mathbb{P}_{n+1}}}\;,\overrightarrow {{\mathbb{P}_{m}}} \right\rangle = {p^2}( {P_{n+m+4}} + {P_{n+1}}\,{P_{m}} ) \\
\quad\quad \quad \quad \quad \quad \quad + p\,q\,({{P_{n+m+3}} + 2\,{P_{n+2}\,{P_{m+1}} + 2\,{P_{n+m+1}}}} ) \\
\quad\quad \quad \quad \quad \quad \quad  + {q ^2}\,({{P_{n+ m+2}} + {P_{n}}\,{P_{m-1}}})\,.
\end{array}
\end{aligned}
\end{equation}
Then from the equations (\ref{4-19}), (\ref{4-20}) and (\ref{4-21}), we have the equation (\ref{4-18}).
\end{proof} 

\noindent\textbf{Special Case-3:} For the dot product of generalized dual Pell vectors $\overrightarrow {{\mathbb{D}^P_n}}$ and $\overrightarrow {{\mathbb{D}^P_{n+1}}}$, we get
\begin{equation} \label{4-22}
\begin{aligned}
\begin{array}{lll}
\left\langle {\overrightarrow {{\mathbb{D}^P_n}}},\, {\overrightarrow {{\mathbb{D}^P_{n+1}}}} \right\rangle&= {\mathbb{D}^P_n}\,{\mathbb{D}^P_{n+1}}+{\mathbb{D}^P_{n+1}}\,{\mathbb{D}^P_{n+2}}+{\mathbb{D}^P_{n+2}}\,{\mathbb{D}^P_{n+3}}\\
&=\left\langle \overrightarrow {{\mathbb{P}_{n}}}
,\overrightarrow {{\mathbb{P}_{n+1}}} \right\rangle+\varepsilon[\, \left\langle \overrightarrow {{\mathbb{P}_{n}}}
,\overrightarrow {{\mathbb{P}_{n+2}}} \right\rangle +\left\langle \overrightarrow {{\mathbb{P}_{n+1}}}
,\overrightarrow {{\mathbb{P}_{n+1}}} \right\rangle  ]\\
& = {p^2}\left\lbrace ( {{P_{2n + 4}} + {P_n}\,{P_{n + 1}}})+\varepsilon\,( 2{P_{2n + 5}}+ {P_{2n + 1}+ 2\,{P_n}\,{P_{n+1}}})\right\rbrace \\
&+p\,q\{ ( 2\,{P_{2n + 3}}+\,{P_{2n - 1}+2\,{P_{n}}\,{P_{n-1}}})\\
&+\varepsilon\,( 2\,{P_{2n+4}}+4\,{P_{2n+2}} + 2\,{P_{2n+3}} + 4\,{P_{n+1}}\,{P_{n+2}} )\\
&+{q^2}\{  ( {{P_{2n + 2}} + {P_{n-1}}\,{P_{n}}})\\
&+\varepsilon ( {2{P_{2n + 3}}+ {P_{2n-1}}+2\,{P_{n}}\,{P_{n-1}}}) \}
\end{array}	
\end{aligned}
\end{equation}
and \\
\begin{equation}\label{4-23}
\begin{aligned}
\begin{array}{lll}
\left\langle {\overrightarrow {{\mathbb{D}^P_n}}},\, {\overrightarrow {{\mathbb{D}^P_{n}}}} \right\rangle & = ({\mathbb{D}^P_n})^2 +({\mathbb{D}^P_{n+1}})^2+({\mathbb{D}^P_{n+2}})^2\\
&=\left\langle \overrightarrow {{\mathbb{P}_{n}}}
, \, \overrightarrow {{\mathbb{P}_{n}}} \right\rangle+2\,\varepsilon \, \left\langle \overrightarrow {{\mathbb{P}_{n}}}
,\, \overrightarrow {{\mathbb{P}_{n+1}}} \right\rangle\\
&  =  {p^2}({{P_{2n + 3}} + P_n^2})+p\,q({P_{2n+2}}+2{P_{n+1}^2}+ 2\,{P_{2n}}) \\
&\quad \quad +{q^2}({P_{2n + 1}} + {P_{n-1}^2}) ]\\
&+2\,\varepsilon \{{p^2}({P_{2n + 4}}+{P_{n}}{P_{n+1}})+p\,q({P_{2n+3}}+2\,{P_{n+1}}\,{P_{n+2}})\\
&\quad \quad +{q^2}({P_{2n + 2}}+{P_{n-1}}{P_{n}})\}.
\end{array}	
\end{aligned}
\end{equation}
\\
We obtain the norm of the generalized dual Pell vector \footnote[3]{Norm of the dual number is:
$$
\left\| {{\overrightarrow {{{\rm\mathbb{A}}}}}} \right\|=\sqrt{a+\varepsilon\,a^*}=\sqrt{a}+\varepsilon\,a^*\frac{1}{2\sqrt{a}}, \,\\ \\
\\
A=a+\varepsilon\,a^* \quad [1].
$$.}
like this; \\ \,
\\
\begin{equation}\label{4-24}	
\begin{aligned}
\begin{array}{lll}
\left\| {{\overrightarrow {{\mathbb{D}^P_n}}}} \right\|& = \sqrt[]{\left[\left\langle \overrightarrow {{\mathbb{D}^P_n}},\, \overrightarrow {{\mathbb{D}^P_n}} \right\rangle\right]}= \sqrt[]{\left[({\mathbb{D}^P_n})^2 +({\mathbb{D}^P_{n+1}})^2+({\mathbb{D}^P_{n+2}})^2\right]} \\

&={\sqrt[]{\left[ {p^2}({P_{2n + 3}} + P_n^2)++p\,q({P_{2n+2}}+2{P_{n+1}^2}+ 2\,{P_{2n}})+{q^2}({P_{2n + 1}} + {P_{n-1}^2}) \right]}} \\
&+\sqrt[]{  2\varepsilon\{{{p^2}({P_{2n+4}}+{P_{n}}{P_{n+1}})+pq({P_{2n+3}}+2{P_{n+1}}{P_{n+2}})}} \\
&+\sqrt[]{{q^2}({P_{2n+2}}+{P_{n-1}}{P_{n}})\}}.
\end{array}
\end{aligned}
\end{equation}
\\	
\noindent\textbf{Special Case-4:} For $p=1,\, q=0$ in the equations (\ref{4-18}), (\ref{4-22}) and (\ref{4-24}), we have \\
\begin{equation*}
	\begin{aligned}
		\begin{array}{lll}
			\left\langle {\overrightarrow {{\mathbb{D}^P_n}}},\, {\overrightarrow {{\mathbb{D}^P_{m}}}} \right\rangle & = \left[P_{n+m+3}+P_{n}P_{m}\right]+\varepsilon \left[2P_{n+m+4}+P_{n+m}+2P_{n}P_{m}\right]
		\end{array}	
	\end{aligned}
\end{equation*}
\begin{equation*}
\begin{aligned}
\begin{array}{lll}
\left\langle {\overrightarrow {{\mathbb{D}^P_n}}},\, {\overrightarrow {{\mathbb{D}^P_{n+1}}}} \right\rangle & = \left[P_{2n+4}+P_{n}P_{n+1}\right]+\varepsilon \left[2P_{2n+5}+P_{2n+1}+2P_{n}P_{n+1}\right]
\end{array}	
\end{aligned}
\end{equation*}

and \\

$\left\| {{\overrightarrow {{D}^P_n}}} \right\| = {\sqrt[]{ ({{P_{2n + 3}} + P_n^2}\,)+2\,\varepsilon \,({P_{2n + 4}}+{P_{n}}{P_{n+1}})}}$ \\

$\quad \quad \quad   =({{P_{2n + 3}} + P_n^2}\,)+\varepsilon \,{\, \frac{({P_{2n + 4}}+{P_{n}}{P_{n+1}})}{\sqrt[]{({P_{2n + 3}} + {P_n}^2\,)} }} .$ \\

\begin{theorem}
Let ${\overrightarrow {{\mathbb{D}^P_n}}}$  and ${\overrightarrow {{\mathbb{D}^P_m}}}$  be two generalized dual Pell vectors. The cross product of ${\overrightarrow {{\mathbb{D}^P_n}}}$  and ${\overrightarrow {{\mathbb{D}^P_m}}}$  is given by
\\
\begin{equation}\label{4-25}
{\overrightarrow {{\mathbb{D}^P_n}}} \times {\overrightarrow {{\mathbb{D}^P_m}}} = {( - 1)^{m+1}}{P_{n - m}}\,(1 + 2\,\varepsilon )\,e_{P}\,(( i + 2\,j - k).
\end{equation}
\end{theorem}
\begin{proof}
The cross product of $\overrightarrow {{\mathbb{D}^P_n}}=\overrightarrow {{\mathbb{P}_{n}}}+\varepsilon \,\overrightarrow {{\mathbb{P}_{n}}}$ \,and \,  $\overrightarrow {{\mathbb{D}^P_m}}=\overrightarrow {{\mathbb{P}_{m}}}+\varepsilon \,\overrightarrow {{\mathbb{P}_{m+1}}}$ \, defined by 
\\
\begin{equation*}
{\overrightarrow {{\mathbb{D}^P_n}}} \times {\overrightarrow {{\mathbb{D}^P_m}}} =({\overrightarrow {{\mathbb{P}_{n}}}} \times {\overrightarrow {{\mathbb{P}_{m}}}})+\varepsilon \,({\overrightarrow {{\mathbb{P}_{n}}}} \times {\overrightarrow {{\mathbb{P}_{m+1}}}}+{\overrightarrow {{\mathbb{P}_{n+1}}}} \times {\overrightarrow {{\mathbb{P}_{m}}}}),
\end{equation*}\\
where ${\overrightarrow {{\mathbb{P}_{n}}}}$ is the generalized Pell vector and ${\overrightarrow {{\mathbb{P}_{n}}}} \times {\overrightarrow {{\mathbb{P}_{m}}}}$ \, is the cross product for the generalized Pell vectors  ${\overrightarrow {{\mathbb{P}_{n}}}}$ and ${\overrightarrow {{\mathbb{P}_{m}}}}$.\\

Now, we calculate the cross products ${\overrightarrow {{\mathbb{P}_{n}}}} \times {\overrightarrow {{\mathbb{P}_{m}}}}$,\, ${\overrightarrow {{\mathbb{P}_{n}}}} \times {\overrightarrow {{\mathbb{P}_{m+1}}}}$ \, and\\ ${\overrightarrow {{\mathbb{P}_{n+1}}}} \times {\overrightarrow {{\mathbb{P}_{m}}}}$:\\

Using the property $P_m P_{n+1}-P_{m+1}P_n=(-1)^n P_{m-n}$, we get 

\begin{equation}\label{4-26}
\begin{aligned}
\begin{array}{ll}
{\overrightarrow {{\mathbb{P}_{n}}}} \times {\overrightarrow {{\mathbb{P}_{m}}}}&=(-1)^{m+1} P_{n-m}(i+2j-k)({p^2} - 2\,pq - {q^2}) \\
\\
&=(-1)^{m+1} P_{n-m}\,(i+2j-k)\,e_{P},
\end{array}
\end{aligned}
\end{equation}
\begin{equation}\label{4-27}
\begin{aligned}
\begin{array}{ll}
{\overrightarrow {{\mathbb{P}_{n}}}} \times {\overrightarrow {{\mathbb{P}_{m+1}}}}&=(-1)^{m+2} P_{n-m-1}\,(i+2j-k)({p^2} - 2\,pq - {q^2}) \\
\\
&=(-1)^{m+2} P_{n-m-1}\,(i+2j-k)\,e_{P}
\end{array}
\end{aligned}
\end{equation}

and \\
\begin{equation}\label{4-28}
\begin{aligned}
\begin{array}{ll}
{\overrightarrow {{\mathbb{P}_{n+1}}}} \times {\overrightarrow {{\mathbb{P}_{m}}}}&=(-1)^{m+1} P_{n-m+1}\,(i+2j-k)({p^2} - 2\,pq - {q^2}) \\
\\
&=(-1)^{m+1} P_{n-m+1}\,(i+2j-k)\,e_{P}.
\end{array}
\end{aligned}
\end{equation} \\ 
Then from the equations (\ref{4-26}), (\ref{4-27}) and (\ref{4-28}), we obtain
the equation (\ref{4-25}).\\ \\ \,
\textbf{Special Case-5:} For $p=1,\,q=0$ in the equation (\ref{4-25}), we have 
$${{\overrightarrow {{D}^P_{n}}} \times {\overrightarrow {{D}^P_{m}}}}=(-1)^{m+1} P_{n-m}\,(1+2\,\varepsilon) (i+2j-k). $$

\end{proof}
\begin{theorem}
Let ${\overrightarrow {{\mathbb{D}^P_n}}}$, ${\overrightarrow {{\mathbb{D}^P_m}}}$  and ${\overrightarrow {{\mathbb{D}^P_k}}}$  be the generalized dual Pell vectors. The mixed product of these vectors is
\begin{equation}\label{4-29}
\left\langle {\overrightarrow {{\mathbb{D}^P_n}}} \times {\overrightarrow {{\mathbb{D}^P_m}}}\;,\;{\overrightarrow {{\mathbb{D}^P_k}}} \right\rangle  = 0 .
\end{equation}
\end{theorem}

\begin{proof}
Using the properties \\

${\overrightarrow {{\mathbb{D}^P_n}}} \times {\overrightarrow {{\mathbb{D}^P_m}}} = ({\overrightarrow {{\mathbb{P}_{n}}}} \times {\overrightarrow {{\mathbb{P}_{m}}}}) + \varepsilon \,({\overrightarrow {{\mathbb{P}_{n}}}} \times {\overrightarrow {{\mathbb{P}_{m+1}}}} + {\overrightarrow {{\mathbb{P}_{n+1}}}} \times {\overrightarrow {{\mathbb{P}_{m}}}})$ \\

and ${\overrightarrow {{\mathbb{D}^P_k}}}=\overrightarrow {{\mathbb{P}_{k}}}+\varepsilon \,\overrightarrow {{\mathbb{P}_{k+1}}}$, we can write\\
$$\begin{aligned}
\begin{array}{lll}
\left\langle {\overrightarrow {{\mathbb{D}^P_n}}} \times {\overrightarrow {{\mathbb{D}^P_m}}},\,\, {\overrightarrow {{\mathbb{D}^P_k}}} \right\rangle= \left\langle {{\overrightarrow {{\mathbb{P}_{n}}}} \times {\overrightarrow {{\mathbb{P}_{m}}}},\,\,{\overrightarrow {{\mathbb{P}_{k}}}}}\right\rangle +
\varepsilon \, [ \left\langle {{\overrightarrow {{\mathbb{P}_{n}}}} \times {\overrightarrow {{\mathbb{P}_{m}}}},\,\,{\overrightarrow {{\mathbb{P}_{k+1}}}}}\right\rangle\\\,\\

+\left\langle {{\overrightarrow {{\mathbb{P}_{n}}}} \times {\overrightarrow {{\mathbb{P}_{m+1}}}},\,\,{\overrightarrow {{\mathbb{P}_{k}}}}}\right\rangle +
\left\langle {{\overrightarrow {{\mathbb{P}_{n+1}}}} \times {\overrightarrow {{\mathbb{P}_{m}}}},\,\,{\overrightarrow {{\mathbb{P}_{k+1}}}}}\right\rangle].
\end{array}	
\end{aligned}
$$
Then from equations (\ref{4-26}), (\ref{4-27}) and (\ref{4-28}), we obtain 
$$
\left\langle (i+2j-k),\,\overrightarrow {{\mathbb{P}_{k}}} \right\rangle= {\mathbb{P}_{k}}+2\,{\mathbb{P}_{k+1}}-{\mathbb{P}_{k+2}}=0,
$$
$$\left\langle (i+2j-k),\,\overrightarrow {{\mathbb{P}_{k+1}}} \right\rangle= {\mathbb{P}_{k+1}}+2\,{\mathbb{P}_{k+2}}-{\mathbb{P}_{k+3}}=0 .$$
Thus, we have the equation (\ref{4-29}).
\end{proof}
\section{Results and Discussion}
\par The generalized Pell, the generalized complex Pell and the generalized dual Pell sequences have been defined as follows:
\\
\par The generalized Pell sequence is given by
\begin{equation}\label{G1}
{\mathbb{P}_{n}} = 2\,{\mathbb{P}_{n-1}} + {\mathbb{P}_{n-2}} \;,\;\;(n \ge 3)
\end{equation}
with ${\mathbb{P}_{0}} = q\,,\;{\mathbb{P}_{1}} = p\,,\;{\mathbb{P}_{2}} = 2\,p + \,q$,\,
where $p,\,q$ are arbitrary integers. \\
\par The generalized complex Pell sequence is given by
\begin{equation}\label{G2}
{\mathbb{C}_{n}} = {\mathbb{P}_{n }} + i \,{\mathbb{P}_{n + 1}}\;,\,\, i^2=-1 ,\,\, i\ne{0}
\end{equation}
with ${\mathbb{C}_1} = p\,+i\,(2\,p + q) \,,\;{{C}_2} = (2\,p\,+\,q)+i\,(5\,p + 2\,q),\;{\mathbb{C}_3} = (5\,p + 2\,q)\,+\,i\,(12\,p + 5\,q)$,...  \,
where $p,q$ are arbitrary integers.\\
\par The generalized dual Pell sequence is given by
\begin{equation}\label{G3}
{\mathbb{D}^P_{n}} = {\mathbb{P}_{n}}+ \varepsilon\,{\mathbb{P}_{n+1}})={\mathbb{D}^P_{n}} = (\,p\,+\,\varepsilon\,(\,2\,p + q\,))\,{P_{n}} + (\,q\,+\,\varepsilon\,p\,){P}_{n - 1}\;,\;\;(n \ge 2).
\end{equation}
with ${\mathbb{D}^P_0} = q\,+\varepsilon(p + q)\,,\;{\mathbb{D}^J_1} = p +\,q\,+\,\varepsilon(p + 3\,q),\;{\mathbb{D}^J_2} = p + 3\,q\,+\,\varepsilon(3\,p + 5\,\,q)$ \,
where $p,q$ are arbitrary integers.

For real, complex and dual cases of the generalized Pell sequence are examined special cases (for ${{p}=1 , \, {q}=0}$). 
In addition, limit of the generalized Pell number and Binet's formula for the generalized Pell sequence are given. Furthermore, real, complex and dual vectors and dot product, cross product and mixed product of these vectors are given.
\section{Conclusions}
\par The generalized Pell, the generalized complex Pell and the generalized dual Pell sequences have been introduced and studied. The use of such special sequences has increased significantly in quantum mechanics, quantum physics, etc.   
\section{Declaration}
\subsection{Competing Interests}The authors declare that they have no competing interests.
\subsection{Funding}The authors received no direct funding for this research.
\subsection{Authors' Contributions}
F.T.A. conceived of the presented idea. F.T.A. developed the theory and performed the computations. F.T.A verified the analytical methods. F.T.A. contributed to the design and implementation of the research, to the analysis of the results. K.K  reviewed the manuscript. All authors discussed the results and contributed to the writing of the manuscript.

\end{document}